\documentclass[12pt,twoside]{amsart}
\usepackage{amsmath, amsthm, amscd, amsfonts, amssymb, graphicx, color}
\usepackage[colorlinks]{hyperref}
\usepackage{cite}

\newtheorem{thm}{Theorem}[section]
\newtheorem{cor}[thm]{Corollary}
\newtheorem{lem}[thm]{Lemma}
\newtheorem{prop}[thm]{Proposition}
\newtheorem{defn}[thm]{Definition}
\newtheorem{rem}[thm]{\bf{Remark}}

\numberwithin{equation}{section}
\def\pn{\par\noindent}

\begin{document}
\vspace{1.3 cm}


\title{On the isomorphism problem for central extensions II}
\author{Noureddine Snanou}
\thanks{{\scriptsize
\hskip -0.4 true cm MSC(2010): Primary: 20J05, 20E22; Secondary: 20J06, 20D40.
}}


\begin{abstract}

In this paper, we study the isomorphism problem for central extensions. More precisely, in some new situations, we provide necessary and sufficient conditions for two central extensions to be isomorphic. We investigate the case when the quotient group is simple or purely non-abelian. Furthermore, we characterize isomorphisms leaving the quotient group invariant. Finally, we deal with isomorphisms of central extensions where the kernel group and the quotient group are isomorphic.

\end{abstract}

\maketitle

\pagestyle{myheadings}
\markboth{\rightline {\sl \hskip 7 cm N. Snanou }}
         {\leftline{\sl  }}

\bigskip


\section{Introduction}	
\vskip 0.4 true cm

The classification of groups in a certain class is one of the most classical problems in group theory. For groups with composition series, the Jordan--H\"{o}lder Theorem states that if we can list all simple groups, and solve the extension problem then we can construct and classify all groups. The classification of simple groups has been achieved in the finite case, hence we need to solve the extension problem. The extension problem for two groups $G_{1}$ and $G_{2}$ is the problem of finding all groups $G$ with $G_{1}$ as a normal subgroup of $G$, and the quotient group $G/G_{1}$ isomorphic to $G_{2}$. Such a group $G$ is called an extension of $G_{1}$ by $G_{2}$ \cite{ROT95}. The classification of extensions with nonabelian kernel group may be found in many texts, but the famous references for these extensions are Schreier's paper \cite{Sch26} and Eilenberg-Mac Lane's paper \cite{EM47II}. In this work, we will focuss on extensions with abelian kernel group. In particular, if $G_{1}$ is a central subgroup of $G$, then we say that $G$ is a central extension of $G_{1}$ by $G_{2}$. For central extensions, an answer to the extension problem has been given by H\"{o}lder and Schreier by using the group cohomology ( see \cite[Theorem 7.59]{ROT95}). However, this answer will not enable us to construct all possible non-isomorphic central extensions of $G_{1}$ by $G_{2}$ (the isomorphism problem). In fact, it is very hard to solve the isomorphism problem, but it has been discussed for some special cases in \cite{Sn20, S-C20, Sn23}. In fact, those results do not concern general isomorphisms, but only those of certain type, namely leaving the kernel group or both the two factors invariant, inducing the identity or a commuting automorphism on the quotient group. In this work, necessary and sufficient conditions for two central extensions of $G_{1}$ by $G_{2}$ to be isomorphic are given in some new situations. More precisely, we study the case when the quotient group is simple or purely non-abelian. Furthermore, we characterize isomorphisms leaving the quotient group invariant, and deal with isomorphisms of central extensions where the kernel group $G_{1}$ and the quotient group $G_{2}$ are isomorphic.

Throughout this paper, we denote by $Z(G)$, $G'$, $Aut(G)$ and $End(G)$, respectively, the center, the derived subgroup, the automorphism group, and the monoid of endomorphisms of $G$. For any two groups $H$ and $K$, let $Hom(H,K)$ denote the set of all homomorphisms from $H$ to $K$.
\section{Central extension}
\vskip 0.4 true cm
In this paper, aspects of group cohomology will be used frequently. Therefore, we recall in this section some basic facts of this theory and fix additional notations and terminology.

Let $1\rightarrow G_{1}{\rightarrow }G{\rightarrow }G_{2}\rightarrow 1$ be a group extension, where for convenience we regard the kernel group $G_{1}$ as a subgroup of $G$ and $G_{2}$ is identified with the quotient group $G/G_{1}$. If $G_{1}$ is a central subgroup of $G$, then we say that $G$ is a central extension of $G_{1}$ by $G_{2}$.
Two central extensions $G$ and $G'$ of $G_{1}$ by $G_{2}$ are equivalent if
and only there is a homomorphism  $\varphi:G\rightarrow G'$ such that the diagram
\begin{equation*}
\begin{array}{ccccccccc}
  1 & \rightarrow & G_{1} & \rightarrow & G & \rightarrow & G_{2} & \rightarrow & 1 \\
    &  & \parallel &  & \text{ \ }\downarrow\varphi &  & \parallel &  &  \\
  1 & \rightarrow & G_{1} & \rightarrow & G' & \rightarrow & G_{2} & \rightarrow & 1
\end{array}
\end{equation*}
commutes.

Let $G_{2}$ be a group which acts trivially on a group $G_{1}$. A 2-cocycle of $G_{2}$ with coefficients in $G_{1}$ is a map $%
\varepsilon :G_{2}\times G_{2}\rightarrow G_{1}$ satisfying the 2-cocycle condition, that is
$$\varepsilon(h,g)\varepsilon (hg,k)=\varepsilon(g,k)\varepsilon (h,gk) \text{ \ for all \ } g, h, k\in G_{2}.$$
We always assume that $\varepsilon$ is normalized, i.e. $\varepsilon(g,1)=\varepsilon(1,g)=1$ for all  $g\in G_{2}$. Note that 2-cocycles are known by factor sets in many books (see for example \cite{EM42, D-F04, BRO82, Mac63, ROT95, W94}).

The set of normalized 2-cocycles of $G_{2}$ with coefficients in $G_{1}$ is denoted by $Z^{2}(G_{2},G_{1})$. The trivial 2-cocycle is the 2-cocycle $c$ with $c(g,h)=1$ for all $g,h\in G_{2}$. Let $\varepsilon_{1}$, $\varepsilon_{2}\in Z^{2}(G_{2},G_{1})$. We write $\varepsilon_{1}\sim\varepsilon_{2}$ and say that $\varepsilon_{1}$ and $\varepsilon_{2}$ are cohomologous, if there is a map $t:G_{2}\to G_{1}$ such that $\varepsilon_{2}(g,h)=t(g)t(h)\varepsilon_{1}(g,h)t(gh)^{-1}$ for all $g$, $h\in G_{2}$. Then $(\sim)$ is an equivalence relation on $Z^{2}(G_{2},G_{1})$. The cohomology class of $\varepsilon\in Z^{2}(G_{2},G_{1})$ is denoted by $[\varepsilon]$. The set of all cohomology classes of $G_{2}$ with coefficients in $G_{1}$ is denoted by $H^{2}(G_{2},G_{1})$ and called the second cohomology of $G_{2}$ with coefficients in $G_{1}$.

From now, $G_{1}$ will always considered an abelian group. Then $Z^{2}(G_{2},G_{1})$ is an abelian group and we have $H^{2}(G_{2},G_{1})=Z^{2}(G_{2},G_{1})/B^{2}( G_{2},G_{1})$ where $B^{2}(G_{2},G_{1})$ is the subgroup of $Z^{2}(G_{2},G_{1})$ which consists of all functions $\psi
:G_{2}\times G_{2}\rightarrow G_{1}$ satisfying that for all $g,h\in G_{2}$:
$\psi (h,g)=\delta ( g)\delta (hg)
^{-1}\delta ( h) $ for some $\delta :G_{2}\rightarrow G_{1}$. The
elements of $B^{2}( G_{2},G_{1}) $ are called 2-coboundaries. The set of all normalized 2-cocycles which are symmetric form a
subgroup of $Z^{2}(G_{2},G_{1})$ and denoted by $SZ^{2}(G_{2},G_{1})$. The famous Schreier theorem says that the central extensions
of $G_{1}$ by $G_{2}$ are classified by the
non-trivial elements of the second cohomology group $H^{2}(G_{2},G_{1})$ with coefficients in $G_{1}$. In particular, a central extension
of $G_{1}$ by $G_{2}$ splits if and only if the corresponding $2$-cocycle is trivial in $H^{2}(G_{2},G_{1})$.

A 2-cocycle $\varepsilon \in Z^{2}(G_{2},G_{1})$ gives rise to a central extension $G=G_{1}\underset
{\varepsilon}{\times }G_{2}$ of $G_{1}$ by $G_{2}$ induced by $\varepsilon$, with group operation given by
\begin{equation*}
( x,y) \underset{\varepsilon}{\bullet }( x',y') =( xx'\varepsilon(y,y'),yy')
\end{equation*}
for all $x$, $x'\in G_{1}$ and $y$, $y'\in G_{2}$. The converse follows from the Schreier theorem, and then a central extension of $G_{1}$ by $G_{2}$ is isomorphic to the group $G_{1}\underset{\varepsilon}{\times }G_{2}$ for some 2-cocycle $\varepsilon \in Z^{2}(G_{2},G_{1})$. We can easily see that the group $G_{1}%
\underset{\varepsilon}{\times }G_{2}$ is abelian if and only if $G_{2}$ is abelian and $\varepsilon
\in SZ^{2}(G_{2},G_{1})$. We know that $G_{1}\underset{\varepsilon}{\times }G_{2} = G_{1}\times G_{2}$ if and only if $\varepsilon=1$. But, it is possible for a central extension of $G_{1}$ by $G_{2}$ induced by a non-trivial 2-cocycle to be isomorphic to the direct product $G_{1}\times G_{2}$ (see Corollaries \ref{Direct1}, \ref{Direct2}).
\section{Preliminary results}
\vskip 0.4 true cm
Let $pr_{i}:G_{1}\underset{\varepsilon}{\times }G_{2}\rightarrow G_{i}$ be the $%
ith$ canonical projection and $t_{i}:G_{i}\rightarrow G_{1}\underset{\varepsilon}{\times }G_{2}$ be the $ith$ canonical injection. Let
$\varphi$ be a group homomorphism from $G_{1}\underset%
{\varepsilon_{1}}{\times }G_{2}$ to $G_{1}\underset{\varepsilon_{2}}{\times }G_{2}$ and set $\varphi
_{ij}=pr_{i}\circ \varphi \circ t_{j}$ where $1\leq i,j\leq 2$. So we can write $\varphi $ in the matrix form: $\varphi =\left(
\begin{array}{cc}
\varphi _{11} & \varphi _{12} \\
\varphi _{21} & \varphi _{22}%
\end{array}%
\right) $. Obviously, we see that $pr_{2}$ and $t_{1}$ are group homomorphisms, then $\varphi _{21}$ is a group homomorphism. Furthermore, we have the following lemmas which we need in the sequel.

\begin{lem}\cite[Lemma 3.1]{S-C20}
\label{class1} Let $\varphi=\left(
\begin{array}{cc}
\varphi _{11} & \varphi _{12} \\
\varphi _{21} & \varphi _{22}%
\end{array}%
\right) $ be a group homomorphism from $G_{1}\underset%
{\varepsilon_{1}}{\times }G_{2}$ to $G_{1}\underset{\varepsilon_{2}}{\times }G_{2}$. Then
\begin{eqnarray}\label{Hom}
\varphi(x,y)=(\varphi_{11}(x)\varphi_{12}( y) \varepsilon _{2}( \varphi
_{21}( x),\varphi_{22}( y)),\text{ }\varphi _{21}(x) \varphi_{22}( y))
\end{eqnarray}
for all $x\in G_{1}$, and $y\in G_{2}$.
\end{lem}

\begin{lem}\cite[Lemma 3.2]{S-C20}
\label{class2}Let $\varphi$ be a set map from $G_{1}\underset{\varepsilon _{1}}{\times
}G_{2}$ to $G_{1}\underset{\varepsilon _{2}}{\times }G_{2}$. Then
$\varphi$ is a group homomorphism if and only if

\begin{eqnarray}\label{Hom3}
\varphi (x,y) \underset{\varepsilon _{2}}{\bullet }
\varphi (x',1)&=&\varphi (xx',y),
\end{eqnarray}
and
\begin{eqnarray}\label{Hom4}
\varphi (x,y) \underset{\varepsilon
_{2}}{\bullet } \varphi ( 1,y')&=&\varphi (
x\varepsilon _{1}(y,y'),yy^{\prime
})
\end{eqnarray}
for all $x$, $x' \in G_{1}$, and $y$, $y'\in G_{2}$.
\end{lem}

\begin{defn}Let $G_{2}$ be a group which acts trivially on an abelian group $G_{1}$. Let $\varepsilon\in Z^{2}(G_{2},G_{1})$. A map $\chi:G_{1}\rightarrow G_{1}$ is called an $\varepsilon$-endomorphism of $G_{1}$, if $$\chi(x\varepsilon(y,y'))=\chi(x)\chi(\varepsilon(y,y'))$$ for all $x\in G_{1}$, and $y$, $y'\in G_{2}$. If in addition $\chi$ is a bijection, then it is said to be $\varepsilon$-automorphism.
\end{defn}

The following lemma follows directly by using the 2-cocycle condition.
\begin{lem}
Let $G_{2}$ be a group which acts trivially on an abelian group $G_{1}$. Let $\varepsilon \in Z^{2}(G_{2},G_{1})$, $\delta\in Hom(G_{1},G_{2})$  and $\sigma$ an $\varepsilon$-endomorphism of $G_{1}$. Then $\sigma \circ \varepsilon\in Z^{2}(G_{2},G_{1})$, $\delta \circ \varepsilon\in Z^{2}(G_{2},G_{2})$ and
$\varepsilon\circ (\delta \times \delta )\in Z^{2}(G_{1},G_{1})$.
\end{lem}

From now, if $\varphi=\left(
\begin{array}{cc}
\varphi _{11} & \varphi _{12} \\
\varphi _{21} & \varphi _{22}%
\end{array}%
\right) $ is a map from $G_{1}\underset{\varepsilon_{1}}{\times }G_{2}$ to $G_{1}\underset{\varepsilon_{2}}{\times }G_{2}$, then $\varphi$ is defined by the formula \eqref{Hom}. From the previous Lemmas, we get the following interesting result which will be frequently used in the sequel.

\begin{prop}
\label{classification0} Let $G_{2}$ be a group such that the equivalence relation $(\sim)$ is trivial on $Z^{2}(G_{2},G_{2})$. Let $\varphi=\left(
\begin{array}{cc}
\varphi _{11} & \varphi _{12} \\
\varphi _{21} & \varphi _{22}%
\end{array}%
\right)$ be a map from $G_{1}\underset{\varepsilon_{1}}{\times }G_{2}$ to $G_{1}\underset{\varepsilon_{2}}{\times }G_{2}$. Then, $\varphi$ is a group homomorphism if and only if
\begin{enumerate}
\item $\varphi_{21}\in Hom(G_{1}, C_{G_{2}}(\varphi_{22}(G_{2})))$, $\varphi _{22}\in End(G_{2})$ and $\varphi_{11}$ is an $\varepsilon_{1}$-endomorphism,
\label{Hom5}
\item $\left[\{1\}\times\varphi_{22}(G_{2}), \{1\}\times\varphi_{21}(G_{1})\right]=1$,
\item $Im(\varepsilon_{1})\leq Ker(\varphi _{21})$ and $\varepsilon_{2}^{-1}\circ(\varphi_{21}\times\varphi_{21})=\psi_{\varphi_{11}}\in B^{2}(G_{1},G_{1})$,  \label{cond0}
\item $(\varphi_{11}\circ\varepsilon_{1})(\varepsilon_{2}^{-1}\circ(\varphi _{22}\times\varphi_{22}))=\psi_{\varphi_{12}}\in B^{2}(G_{2},G_{1})$. \label{Hom6}
\end{enumerate}
where $\psi_{\varphi_{ij}}(y,y')=\varphi_{ij}(y) \varphi_{ij}(y')\varphi_{ij}(yy')^{-1}$ for all $1\leq i, j\leq2$ and $y$, $y'\in G_{j}$.
\end{prop}

\begin{proof} Indeed, evaluate the left hand side and right hand side of the formulas \eqref{Hom3} and \eqref{Hom4}, we obtain

\begin{align}&\text{ \ \ }\varphi(x,y) \underset{\varepsilon_{2}}{\cdot}\varphi (x',1)=\varphi (xx',y) \nonumber\\
&\quad \Leftrightarrow (\varphi_{11}(x)\varphi _{12}(y) \varepsilon_{2}( \varphi
_{21}( x),\varphi_{22}( y)),\text{ }\varphi _{21}(x)\varphi_{22}( y))\underset{\varepsilon_{2}}{\cdot}(\varphi_{11}(x'),\varphi_{21}(x'))\nonumber\\
&\quad\text{ \ \ }=(\varphi_{11}(xx')\varphi _{12}(y) \varepsilon_{2}( \varphi
_{21}( xx'),\varphi_{22}(y)),\text{ }\varphi _{21}( xx') \varphi _{22}(y)) \nonumber\\
\label{Eq2}&\quad \Leftrightarrow \varphi_{21}(x)\varphi_{22}(y)\varphi _{21}(x')=\varphi_{21}(xx') \varphi _{22}(y) \text{ and }\\
\label{Eq3}&\quad \text{ \ \ \ } \varphi_{11}(x)\varphi _{12}(y)\varepsilon_{2}(\varphi_{21}(x),\varphi_{22}(y)) \varphi_{11}(x') \varepsilon_{2}(\varphi_{21}(x)\varphi_{22}(y),\varphi_{21}(x'))\\
&\quad \text{ \ \ }=\varphi _{11}(xx') \varphi _{12}(y) \varepsilon_{2}(\varphi_{21}(xx'),\varphi_{22}(y))\nonumber.
\end{align}
Setting $x=1$ in the equations \eqref{Eq2} and \eqref{Eq3}, we obtain
\begin{eqnarray}\label{Eq4}
\varphi_{22}(y)\varphi_{21}(x')=\varphi_{21}(x')\varphi_{22}(y)
\end{eqnarray}
and
\begin{eqnarray}\label{Eq5}
\varepsilon_{2}(\varphi_{22}(y),\varphi_{21}(x'))=\varepsilon_{2}(\varphi_{21}(x'),\varphi_{22}(y)).
\end{eqnarray}
That is, $\left[\{1\}\times\varphi_{22}(G_{2}), \{1\}\times\varphi_{21}(G_{1})\right]=1$. Now, combining the equations \eqref{Eq2} and \eqref{Eq4}, we get that  $\varphi _{21}\in Hom(G_{1}, C_{G_{2}}(\varphi_{22}(G_{2})))$. Thus, using the 2-cocycle condition together with the equations \eqref{Eq4} and \eqref{Eq5}, the equation \eqref{Eq3} yields
\begin{eqnarray}\label{Eq6}
\varphi_{11}(xx')=\varphi _{11}(x)\varphi_{11}( x')\varepsilon_{2}(\varphi_{21}(x),\varphi_{21}(x'))
\end{eqnarray}
which implies that $\varepsilon_{2}^{-1}\circ(\varphi_{21}\times\varphi_{21})\in B^{2}(G_{1},G_{1})$.

On the other hand, we have that
\begin{align}&\text{ \ \ }\varphi(x,y)\underset{\varepsilon_{2}}{\cdot}\varphi(1,y')=\varphi (x\varepsilon_{1}(y,y'),yy')\nonumber\\
&\quad \Leftrightarrow(\varphi_{11}(x) \varphi _{12}(y) \varepsilon _{2}( \varphi
_{21}(x),\varphi_{22}(y)),\text{ }\varphi _{21}(x) \varphi _{22}(y))\underset{\varepsilon_{2}}{\cdot}(\varphi_{12}(y'),\varphi_{22}(y'))\nonumber\\
&\quad \text{ \ \ \ }=(\varphi_{11}(x\varepsilon_{1}(y,y')) \varphi _{12}(yy') \varepsilon _{2}( \varphi
_{21}( x\varepsilon_{1}(y,y')),\varphi_{22}( yy')),\text{ }\varphi _{21}( x\varepsilon_{1}(y,y')) \varphi _{22}( yy'))\nonumber\\
\label{Eq7}&\quad \Leftrightarrow \varphi _{21}(x) \varphi _{22}(y) \varphi_{22}(y')=\varphi _{21}(x\varepsilon_{1}(y,y')) \varphi _{22}( yy') \text{ and }\\
\label{Eq8}&\quad \text{ \ \ } \varphi_{11}(x)\varphi _{12}(y)\varepsilon_{2}( \varphi_{21}(x),\varphi_{22}( y))\varphi _{12}(y')\varepsilon_{2}( \varphi_{21}( x) \varphi _{22}(y),\varphi_{22}(y'))\\
&\quad \text{ \ \ \ }=\varphi_{11}(x\varepsilon_{1}(y,y'))\varphi _{12}(yy')\varepsilon _{2}( \varphi
_{21}( x\varepsilon_{1}(y,y')),\varphi_{22}(yy'))\nonumber.
 \end{align}

 Since $\varphi_{21}$ is a group homomorphism, the equation \eqref{Eq7} implies that $$\varphi_{21}(\varepsilon_{1}(y,y'))=\varphi _{22}(y) \varphi_{22}(y')(\varphi _{22}( yy'))^{-1}$$ for all $y$, $y'\in G_{2}$. So, $\varphi _{21}\circ\varepsilon_{1}\sim 1$. But, by the assumption, there is no nontrivial 2-cocycle in $Z^{2}(G_{2},G_{2})$ that is cohomologous to the trivial 2-cocycle. So, we have $\varphi _{21}\circ\varepsilon_{1}=1$ and then $\varphi_{22}\in End(G_{2})$. Furthermore, by using the 2-cocycle condition, the equation \eqref{Eq8} gives us
 \begin{eqnarray}\label{Eq9}
\varphi_{11}(x\varepsilon_{1}(y,y'))\varphi _{12}(yy')=\varphi_{11}(x)\varphi _{12}(y) \varphi _{12}(y') \varepsilon_{2}(\varphi _{22}(y),\varphi_{22}(y')).
\end{eqnarray}
 But, the equation \eqref{Eq6} yields $\varphi_{11}(x\varepsilon_{1}(y,y'))=\varphi_{11}(x)\varphi _{11}(\varepsilon_{1}(y,y'))$. Thus, the equation \eqref{Eq9} is equivalent to $$\varphi _{11}(\varepsilon_{1}(y,y'))\varphi _{12}(yy')=\varphi _{12}(y) \varphi_{12}(y') \varepsilon_{2}(\varphi _{22}(y),\varphi_{22}(y')),$$ which implies that $$\varphi _{11}(\varepsilon_{1}(y,y'))(\varepsilon_{2}(\varphi _{22}(y),\varphi_{22}(y')))^{-1}=\varphi _{12}(y) \varphi _{12}(y')(\varphi _{12}(yy'))^{-1}.$$
Thus, the proof is completed.
\end{proof}

\begin{rem} \begin{enumerate}
              \item Note that if $G_{2}$ is abelian then the assumption of the previous proposition and the condition $B^{2}(G_{2},G_{2})=1$ are equivalent.
\item If $\varepsilon_{2}\in SZ^{2}(G_{2},G_{1})$ then the second condition of the previous proposition is not required.
            \end{enumerate}
\end{rem}

\begin{defn}
The groups $G_{1}\underset{\varepsilon_{1}}{\times }G_{2}$ and $G_{1}\underset{\varepsilon_{2}}{\times }G_{2}$ are called upper isomorphic, if there exists an
isomorphism $\varphi :G_{1}\underset{\varepsilon_{1}}{\times }G_{2}\longrightarrow
G_{1}\underset{\varepsilon_{2}}{\times }G_{2}$ leaving $G_{1}$ invariant.
\end{defn}

In particular, suppose that $Z(G_{1}\underset{\varepsilon_{1}}{\times }G_{2})=Z(G_{1}\underset{\varepsilon_{2}}{\times }G_{2})=G_{1}$ or $(G_{1}\underset{\varepsilon_{1}}{\times }G_{2})'=(G_{1}\underset{\varepsilon_{2}}{\times }G_{2})'=G_{1}$. So each isomorphism $\varphi=\left(
\begin{array}{cc}
\varphi _{11} & \varphi _{12} \\
\varphi _{21} & \varphi _{22}%
\end{array}%
\right)$ between  $G_{1}\underset{\varepsilon_{1}}{\times }G_{2}$ and $G_{1}\underset{\varepsilon_{2}}{\times }G_{2}$ leaves $G_{1}$ invariant and then $\varphi_{21}=1$. So, the equation
\eqref{Eq7} implies that $\varphi_{22}\in End(G_{2})$ and then the assumption of the previous proposition is not required. This case is covered by the following result which can be viewed as a consequence of \cite[Theorem 3.7]{Sn20}.
\begin{prop}
\label{Main}The groups $G_{1}\underset{\varepsilon_{1}}{\times }G_{2}$ and $G_{1}\underset{\varepsilon_{2}}{\times }G_{2}$ are upper isomorphic
if and only if there exist $%
\sigma \in Aut(G_{1})$ and $\rho \in Aut(G_{2})$ such that
$$(\sigma\circ\varepsilon_{1})(\varepsilon_{2}^{-1}\circ(\rho \times \rho )) \in B^{2}( G_{2},G_{1}).$$
\end{prop}
The following consequence is straightforward.
\begin{cor}\label{Direct1}
Let $\varepsilon\in Z^{2}(G_{1},G_{1})$. The groups $G_{1}\underset{\varepsilon}{\times }G_{2}$ and $G_{1}\times G_{2}$ are upper isomorphic
if and only if there exist $%
\sigma \in Aut(G_{1})$ such that
$\sigma\circ\varepsilon \in B^{2}( G_{2},G_{1})$.
\end{cor}
\begin{cor}\label{Direct2}
Let $\varepsilon\in Z^{2}(G_{1},G_{1})$. The groups $G_{1}\times G_{2}$ and $G_{1}\underset{\varepsilon}{\times }G_{2}$ are upper isomorphic
if and only if there exist $%
\rho \in Aut(G_{2})$ such that $\varepsilon_{2}^{-1}\circ(\rho \times \rho ) \in B^{2}( G_{2},G_{1})$.
\end{cor}
\section{Central extensions with simple or purely non-abelian quotient group}

\begin{prop}
 Let $G_{2}$ be a simple non-abelian group which acts trivially on an abelian group $G_{1}$. The groups $G_{1}\underset{\varepsilon_{1}}{\times }G_{2}$ and $G_{1}\underset{\varepsilon_{2}}{\times }G_{2}$ are isomorphic if and only if they are upper isomorphic.
\end{prop}

\begin{proof}
The if direction is clear. For the converse, assume that $G_{1}\underset{\varepsilon_{1}}{\times }G_{2}$ and $G_{1}\underset{\varepsilon_{2}}{\times }G_{2}$ are isomorphic by an isomorphism $\varphi=\left(
\begin{array}{cc}
\varphi _{11} & \varphi _{12} \\
\varphi _{21} & \varphi _{22}%
\end{array}%
\right)$. It follows from the equation $\varphi (x,1) \underset{\varepsilon
_{2}}{\bullet }\varphi ( 1,y)=\varphi (1,y) \underset{\varepsilon
_{2}}{\bullet }\varphi ( x,1)$ that $[\varphi _{21}(x),\varphi _{22}(y)]=1$ for all $x\in G_{1}$, and $y\in G_{2}$. Now let $g\in G_{2}$, there exists an element $(x,y)\in G_{1}\underset{\varepsilon_{1}}{\times }G_{2}$ such that $\varphi(x,y)=(1,g)$, that is $g=\varphi_{21}(x)\varphi_{22}(y)$. So, for all $h\in G_{1}$, we have
\begin{eqnarray*}
g\varphi_{21}(h)g^{-1}&=& \varphi_{21}(x)\varphi_{22}(y)\varphi_{21}(h)\varphi_{22}(y)^{-1}\varphi_{21}(x)^{-1}\\
&= &\varphi_{21}(x)\varphi_{21}(h)\varphi_{21}(x)^{-1}\in \varphi_{21}(G_{1}).
\end{eqnarray*}
Thus, $\varphi_{21}(G_{1})$ is a normal subgroup of $G_{2}$. As $G_{2}$ is simple non-abelian, then $\varphi_{21}(G_{1})$ is either trivial or $G_{2}$. If $\varphi_{21}(G_{1})=G_{2}$ then $\varphi_{21}$ is an epimorphism and therefore $G_{2}$ is abelian, a contradiction. Hence, $\varphi_{21}=1$ and then $\varphi$ maps $G_{1}$ to itself, as required.
\end{proof}

 Recall that a non-abelian group which has no non-trivial abelian direct factor is said to be purely non-abelian.
\begin{thm}
 Let $G_{2}$ be a finite purely non-abelian group which acts trivially on a finite abelian group $G_{1}$. Let $\varphi\in\left. \left\{
\begin{pmatrix}
\sigma & \eta \\
\delta & \rho
\end{pmatrix}
\right| \begin{array}{c}
          \sigma\in Aut(G_{1}), \eta\in Hom(G_{2},G_{1}) \\
          \delta\in Hom(G_{1},G_{2}), \rho\in Aut(G_{2})
        \end{array} \right\}$
where $\sigma$, $\delta$, $\rho$ satisfy the conditions:
\begin{enumerate}
\item \label{condition8}$\left[\{1\}\times G_{2}, \{1\}\times\delta(G_{1})\right]=1$,
\item \label{condition0} $\varepsilon_{2}\circ(\delta\times\delta)=1$ \text{ \ and } $\delta\circ\varepsilon_{1}=1$,
\item $\sigma\circ\varepsilon_{1}=\varepsilon_{2}\circ(\rho\times\rho).$ 
\end{enumerate}
Then, $\varphi$ is an isomorphism from $G_{1}\underset{\varepsilon_{1}}{\times }G_{2}$ to $G_{1}\underset{\varepsilon_{2}}{\times }G_{2}$.
\end{thm}
\begin{proof}
Indeed, the map $\varphi$ is defined by the formula \eqref{Hom}. By Proposition \ref{classification0}, the map $\varphi$ is clearly a group homomorphism. Now, assume that $\varphi(x,y)=1$. So $\delta(x)\rho(y)=1$ and then $\rho(y)=\delta(x^{-1})$, which implies that $\sigma(x)\eta(y)\varepsilon_{2}( \delta(x),\delta(x^{-1}))=1$. So, the first equation of the condition (\ref{condition0})  ensures that $\sigma(x)\eta(y)=1$, and then $x= \sigma^{-1}(\eta(y^{-1}))$. Hence, $\rho^{-1}(\delta(\sigma^{-1}(\eta(y))))=y$. From the condition (\ref{condition8}), we have $\left[G_{2},\delta(G_{1})\right]=1$, that is $\delta(G_{1})\leq Z(G_{2})$. Hence, we have  $\Psi=\rho^{-1}\circ\delta\circ\sigma^{-1}\circ\eta\in Hom(G_{2},Z(G_{2}))$, and then $Im\Psi\trianglelefteq G_{2}$. So, by Fitting's Lemma, we have $G_{2}\cong Ker\Psi\times Im\Psi$ which contradicts to the fact that $G_{2}$ is purely non-abelian. Thus $y=1$ and then $x=1$. Therefore, the map $\varphi$ is injective, and then it is an isomorphism.
\end{proof}

\begin{rem}
The previous proposition will not be true if $G_{2}$ is not purely non-abelian. Indeed, assume that $G_{1}$ is a direct factor of $G_{2}$ and let $\varphi=\left(
\begin{array}{cc}
id_{G_{1}} & \varphi _{12} \\
\varphi _{21} & id_{G_{2}}%
\end{array}%
\right)$ be a map from $G_{1}\underset{\varepsilon_{1}}{\times }G_{2}$ to $G_{1}\underset{\varepsilon_{2}}{\times }G_{2}$ where $\varphi_{12}(x)=\varphi_{21}(x)=x^{-1}$ for all $x\in G_{1}$. So, using formula \eqref{Hom}, we obtain $\varphi(x,x)=(\varepsilon_{2}( \varphi_{21}(x),\varphi_{21}(x^{-1})),1)$. Hence, by using the first equation of the condition (\ref{condition0}), we get $\varphi(x,x)=(1,1)$. Therefore, $\varphi$ is not an isomorphism.
\end{rem}

\section{Lower isomorphism problem for central extensions}

\begin{defn}Let $G_{2}$ be a group which acts trivially on an abelian group $G_{1}$. The groups $G_{1}\underset{\varepsilon_{1}}{\times }G_{2}$ and $G_{1}\underset{\varepsilon_{2}}{\times }G_{2}$ are called lower isomorphic, if there exists an
isomorphism $\varphi:G_{1}\underset{\varepsilon_{1}}{\times }G_{2}\longrightarrow
G_{1}\underset{\varepsilon_{2}}{\times }G_{2}$ leaving $G_{2}$ invariant.
\end{defn}

We now present the following main result of this section.

\begin{thm}
\label{main0} Let $G_{2}$ be a group such that the equivalence relation $(\sim)$ is trivial on $Z^{2}(G_{2},G_{2})$. If the groups $G_{1}\underset{\varepsilon_{1}}{\times }G_{2}$ and $G_{1}\underset{\varepsilon_{2}}{\times }G_{2}$ are lower isomorphic then there exist $\rho \in Aut(G_{2})$, $\delta\in Hom(G_{1},Z(G_{2}))$ and an $\varepsilon_{1}$-automorphism $\sigma$ such that
\begin{enumerate}
\item $\left[\{1\}\times G_{2}, \{1\}\times\delta(G_{1})\right]=1$, $Im(\varepsilon_{1})\leq Ker(\delta)$,\label{condition1}
\item $\varepsilon_{2}^{-1}\circ(\delta\times\delta)=\psi_{\sigma}\in B^{2}(G_{1},G_{1})$ where $\psi_{\sigma}(x,x')=\sigma(x) \sigma(x')\sigma(xx')^{-1}$ for all $x$, $x'\in G_{1}$,\label{condition2}
\item $\varepsilon_{2}\circ(\rho\times\rho)=\sigma\circ\varepsilon_{1}$.\label{condition3}
\end{enumerate}
\end{thm}

\begin{proof} Suppose that $G_{1}\underset{\varepsilon_{1}}{\times }G_{2}$ and $G_{1}\underset{\varepsilon_{2}}{\times }G_{2}$ are isomorphic by an isomorphism $\varphi=\left(
\begin{array}{cc}
\varphi _{11} & 1 \\
\varphi _{21} & \varphi _{22}%
\end{array}%
\right)$. From Lemma \ref{class1}, we have that $$\varphi(x,y)=(\varphi_{11}(x)\varepsilon _{2}( \varphi
_{21}(x),\varphi_{22}(y)),\text{ }\varphi _{21}(x) \varphi_{22}( y))$$ for all $x\in G_{1}$, $y\in G_{2}$. Since $\varphi $ is bijective, then so is $\varphi_{22}$. By Proposition \ref{classification0}, the maps $\varphi_{21}\in Hom(G_{1},Z(G_{2}))$, $\varphi _{22}\in Aut(G_{2})$ and the $\varepsilon_{1}$-endomorphism $\varphi_{11}$ satisfy the conditions (\ref{condition1}-\ref{condition3}). So, it remains to show that $\varphi_{11}$ is bijective. Let $g\in G_{1}$, since $\varphi $ is surjective, there exists an element  $(x,y)\in G_{1}\underset{\varepsilon_{1}}{\times }G_{2}$ such that $\varphi(x,y)=(g,1)$, that is $\varphi_{11}(x)\varepsilon _{2}( \varphi
_{21}(x),\varphi_{22}(y))=g$ and $\varphi _{21}(x) \varphi_{22}( y)=1$. So, $\varphi_{11}(x)\varepsilon_{2}( \varphi_{21}(x),\varphi_{21}(x^{-1}))=g$ and then, using the condition (\ref{condition2}), we have $\varphi_{11}(x)=g$. Therefore, $\varphi_{11}$ is surjective. On the other hand, the map $\psi$ defined by $\psi (x,y)=(x,y\varphi _{22}^{-1}(\varphi_{21}(x)^{-1}))$ is a bijection and we have $\varphi \circ \psi(x^{-1},1)=(\varphi_{11}(x^{-1})\varepsilon_{2}(\varphi_{21}(x^{-1}),\varphi_{21}(x)),1)$ for all $x\in G_{1}$. Thus, the condition (\ref{condition2}) ensures that $\varphi \circ \psi(x^{-1},1)=(\varphi_{11}(x)^{-1},1)$ for all $x\in G_{1}$. Since $\varphi \circ \psi$ is injective, then so is $\varphi_{11}$. Thus, the desired result follows directly by taking $\rho=\varphi _{22}$, $\sigma =\varphi _{11}$ and $\delta=\varphi_{21}$.
\end{proof}

\begin{rem} The converse of the previous result holds if $G_{1}$ and $G_{2}$ are finite. Indeed, it suffices to show that $\varphi$ is injective. Let $(x,y)\in G_{1}\times
_{\varepsilon}G_{2}$ such that $\varphi(x,y)=(1,1)$. Then, the equation $\varphi _{21}(x) \varphi_{22}(y)= 1$ implies that $\varphi_{22}(y)=\varphi _{21}(x^{-1})$. So $\varphi_{11}(x)\varepsilon_{2}( \varphi
_{21}(x),\varphi_{21}(x^{-1}))=1$. Using the condition (\ref{condition2}), we get $\varphi_{11}(x^{-1})^{-1}=1$ and then $x=1$ since $\varphi_{11}$ is injective. Since $\varphi _{21}(1)= 1$, it follows that $\varphi_{22}(y)= 1$ and then $y=1$ since $\varphi_{22}$ is injective. Therefore, $\varphi$ is bijective and then it is a lower isomorphism by Proposition \ref{classification0}. As required.
\end{rem}

As direct consequences of Theorem \ref{main0}, we derive

\begin{cor}\label{Direct3}
 Let $\varepsilon\in Z^{2}(G_{2},G_{1})$. The groups $G_{1}\underset{\varepsilon}{\times }G_{2}$ and $G_{1}\times G_{2}$ are lower isomorphic if and only if $\varepsilon=1$.
\end{cor}
\begin{proof}
The if direction is clear. Conversely, taking $\varepsilon_{1}=\varepsilon$ and $\varepsilon_{2}=1$. So from Proposition \ref{main0}, we get that $\sigma\in Aut(G_{1})$ and $Im(\varepsilon)\leq Ker(\sigma)$ and then we must have $\varepsilon=1$. As required.               
\end{proof}

\begin{cor}\label{Direct4}
Let $\varepsilon\in Z^{2}(G_{1},G_{1})$. The groups $G_{1}\times G_{2}$ and $G_{1}\underset{\varepsilon}{\times }G_{2}$ are lower isomorphic
if and only if there exist $\rho \in Aut(G_{2})$ such that $\varepsilon\circ(\rho \times \rho )=1$.
\end{cor}

\begin{proof}
The only if direction follows directly from Proposition \ref{main0} by taking $\varepsilon_{2}=\varepsilon$ and $\varepsilon_{1}=1$. For the converse, the bijection $\varphi$ defined by $\varphi(x,y)=(x,\text{ }\rho(y))$ is an isomorphism.               
\end{proof}

\begin{cor}
  Further to the assumption of the previous theorem, suppose that $B^{2}(G_{1},G_{1})=1$. The groups $G_{1}\underset{\varepsilon_{1}}{\times }G_{2}$ and $G_{1}\underset{\varepsilon_{2}}{\times }G_{2}$ are lower isomorphic if and only if there exist $\rho \in Aut(G_{2})$ and $\sigma \in Aut(G_{1})$ such that  $\varepsilon_{2}\circ(\rho\times\rho)=\sigma\circ\varepsilon_{1}$.
\end{cor}
\begin{proof}
Indeed, suppose that the groups $G_{1}\underset{\varepsilon_{1}}{\times }G_{2}$ and $G_{1}\underset{\varepsilon_{2}}{\times }G_{2}$ are lower isomorphic. By Proposition \ref{main0}, there exist $\rho \in Aut(G_{2})$, $\delta\in Hom(G_{1},Z(G_{2}))$ and an $\varepsilon_{1}$-automorphism $\sigma$ satisfying the conditions (\ref{condition2}) and (\ref{condition3}). Since $B^{2}(G_{1},G_{1})=1$, the condition (\ref{condition2}) implies that $\sigma \in Aut(G_{1})$. Therefore, we conclude by using the condition (\ref{condition3}). For the converse, we can easily prove that the bijection $\varphi$ defined by $\varphi(x,y)=(\sigma(x),\text{ }\rho(y))$ is an isomorphism. Hence the corollary follows.
\end{proof}

\section{Isomorphisms of central extension with isomorphic factors group}

\begin{defn}
Let $G_{2}$ be an abelian group which acts trivially on an abelian group $G_{1}$. The groups $G_{1}\underset{\varepsilon _{1}}{\times }G_{2}$ and $%
G_{1}\underset{\varepsilon _{2}}{\times }G_{2}$ are called
$(G_{i})$-isomorphic, if there exists an isomorphism $\varphi=\left(
\begin{array}{cc}
\varphi _{11} & \varphi _{12} \\
\varphi _{21} & \varphi _{22}%
\end{array}%
\right) $  between $G_{1}%
\underset{\varepsilon _{1}}{\times }G_{2}$ and $G_{1}\underset{%
\varepsilon _{2}}{\times }G_{2}$, such that $\varphi_{ii}=1$. A $(G_{1},G_{2})$-isomorphic is an isomorphism that is both $(G_{1})$-isomorphism and $(G_{2})$-isomorphism.
\end{defn}

\begin{prop}
\label{MAIN1} Suppose that the groups $G_{1}\underset{\varepsilon_{1}}{\times }G_{2}$ and $G_{1}\underset{\varepsilon_{2}}{\times }G_{2}$ are $(G_{2})$-isomorphic. Then, there exist an $\varepsilon_{1}$-endomorphism $\sigma$ of $G_{1}$, an injective map $\eta:G_{2}\rightarrow G_{1}$ and an epimorphism $\delta:G_{1}\rightarrow G_{2}$ such that:

\begin{enumerate}
  \item $\varepsilon_{2}^{-1}\circ(\delta\times\delta)=\psi_{\sigma}\in B^{2}(G_{1},G_{1}),\text{\ }Im(\varepsilon_{1})\leq Ker(\delta)$,
  \item $\sigma\circ\varepsilon_{1}=\psi_{\eta}\in B^{2}(G_{2},G_{1}).$
\end{enumerate}
\end{prop}
\begin{proof}
Let  $\varphi=\left(
\begin{array}{cc}
\varphi _{11} & \varphi _{12} \\
\varphi _{21} & \varphi _{22}%
\end{array}%
\right)$ be a $(G_{2})$-isomorphism from $G_{1}\underset{\varepsilon_{1}}{\times }G_{2}$ to $G_{1}\underset{\varepsilon_{2}}{\times }G_{2}$. So $\varphi _{22}=1$, and therefore we can show easily that  $\varphi_{21}$ is surjective and $\varphi_{12}$ is injective. Hence, by Proposition \ref{classification0}, the desired conditions follows directly by taking $\sigma=\varphi _{11}$,  $\delta=\varphi _{21}$ and $\eta=\varphi _{12}$.
\end{proof}

Note that if $G_{2}$ is non-abelian then the condition $\delta\in Epi(G_{1},G_{2})$ implies that $\delta=1$. Therefore, the previous result becomes a direct consequence of Corollary \ref{Main}.

\begin{cor}
Suppose that $G_{1}$ and $G_{2}$ are two finite abelian groups with the same order. The groups $G_{1}\underset{\varepsilon_{1}}{\times }G_{2}$ and $G_{1}\underset{\varepsilon_{2}}{\times }G_{2}$ are $(G_{2})$-isomorphic if and only if $\varepsilon_{1}=1$ and there exists an isomorphism $\delta:G_{1}\rightarrow G_{2}$ such that $\varepsilon_{2}^{-1}\circ(\delta\times\delta)\in B^{2}(G_{1},G_{1})$.
\end{cor}
\begin{proof}
Indeed, suppose that the groups $G_{1}\underset{\varepsilon_{1}}{\times }G_{2}$ and $G_{1}\underset{\varepsilon_{2}}{\times }G_{2}$ are $(G_{2})$-isomorphic. By the previous proposition, there exists an epimorphism $\delta:G_{1}\rightarrow G_{2}$ such that $Im(\varepsilon_{1})\leq Ker(\delta)$ and $\varepsilon_{2}^{-1}\circ(\delta\times\delta)\in B^{2}(G_{1},G_{1})$. But $\delta$ is in fact an isomorphism by the assumption, so we must have $\varepsilon_{1}=1$. Conversely, since $\varepsilon_{2}^{-1}\circ(\delta\times\delta)\in B^{2}(G_{1},G_{1})$, it follows that there exists a map $\sigma:G_{1}\rightarrow G_{1}$ such that $\varepsilon_{2}^{-1}(\delta(x),\delta(x'))=\sigma(x)\sigma(x')\sigma(xx')^{-1}$ for all $x$, $x'\in G_{1}$. Let $\delta'$ be the inverse of $\delta$. By the normalization condition, we have $\sigma(1)=1$. So, the bijection $\varphi$ defined by $\varphi(x,y)=(\sigma(x)\delta'(y),\text{ }\delta(x))$ is clearly an isomorphism. As required.
\end{proof}

\begin{prop}\label{main2}
Let $G_{2}$ be a group such that the equivalence relation $(\sim)$ is trivial on $Z^{2}(G_{2},G_{2})$. Suppose that the groups $G_{1}\underset{\varepsilon_{1}}{\times }G_{2}$ and $G_{1}\underset{\varepsilon_{2}}{\times }G_{2}$ are $(G_{1})$-isomorphic. Then, there exist $\rho\in End(G_{2})$, a surjective map $\eta:G_{2}\rightarrow G_{1}$ and a monomorphism $\delta:G_{1}\rightarrow G_{2}$ such that
\begin{enumerate}
  \item $\varepsilon_{1}=1$,\label{Cond1}
  \item $\varepsilon_{2}\circ(\delta\times\delta)=1$ and \label{Cond2}
  \item $\varepsilon_{2}^{-1}\circ(\rho\times\rho)=\psi_{\eta}\in B^{2}(G_{2},G_{1})$.\label{Cond3}
\end{enumerate}
\end{prop}
\begin{proof}
Let  $\varphi=\left(
\begin{array}{cc}
1 & \varphi _{12} \\
\varphi _{21} & \varphi _{22}%
\end{array}%
\right)$ be an isomorphism from $G_{1}\underset{\varepsilon_{1}}{\times }G_{2}$ to $G_{1}\underset{\varepsilon_{2}}{\times }G_{2}$. Hence, by taking $\rho=\varphi _{22}$, $\eta =\varphi _{12}$ and $\delta=\varphi_{21}$, the conditions (\ref{Cond2}) and (\ref{Cond3}) follows directly from Proposition \ref{classification0}. Notice that
$\varphi(x,y)=(\varphi_{12}(y)\varepsilon _{2}( \varphi_{21}(x),\varphi_{22}(y)),\text{ }\varphi _{21}(x) \varphi_{22}( y))$ for all $x\in G_{1}$, $y\in G_{2}$. So, $\varphi(x,1)=(1,\text{ }\varphi _{21}(x))$ and then $\varphi _{21}$ is injective. But, by the condition (\ref{cond0}) of Proposition \ref{classification0}, we have $Im(\varepsilon_{1})\leq Ker(\varphi _{21})$ which implies that $\varepsilon_{1}=1$. On the other hand, let $g\in G_{1}$, then there exists $(x,y)\in G_{1}\underset{\varepsilon_{1}}{\times }G_{2}$ such that $\varphi(x,y)=(g,1)$. This gives us $\varphi_{12}(y)\varepsilon _{2}(\varphi_{21}(x),\varphi_{21}(x^{-1}))=g$. So the condition (\ref{Cond2})  ensures that $\varphi_{12}(y)=g$ and then $\varphi_{12}$ is surjective. As required.
\end{proof}

Now, we derive the following consequences. 

\begin{cor}
 Further to the assumption of the previous proposition, suppose that $G_{1}$ and $G_{2}$ are two finite abelian groups with the same order. The groups $G_{1}\underset{\varepsilon_{1}}{\times }G_{2}$ and $G_{1}\underset{\varepsilon_{2}}{\times }G_{2}$ are $(G_{1})$-isomorphic if and only if $\varepsilon_{1}=1$ and there exists an isomorphism $\delta:G_{1}\rightarrow G_{2}$ such that $\varepsilon_{2}\circ(\delta\times\delta)=1$.
\end{cor}

\begin{proof}
Indeed, using the assumptions, the only if direction comes immediately from the previous result. Conversely, let $\delta'$ be the inverse of $\delta$. Define a bijective map $\varphi$ between $G_{1}\times G_{2}$ and $G_{1}\underset{\varepsilon_{2}}{\times }G_{2}$ given by
$\varphi(x,y)=(\delta'(y),\delta(x))$, for all $x\in G_{1}$, $y\in G_{2}$. Since $\varepsilon_{2}\circ(\delta\times\delta)=1$, it is easy to check that $\varphi$ is a
group homomorphism, and therefore it is a group isomorphism.
\end{proof}

\begin{cor}
 The groups $G_{1}\underset{\varepsilon_{1}}{\times }G_{2}$ and $G_{1}\underset{\varepsilon_{2}}{\times }G_{2}$ are $(G_{1},G_{2})$-isomorphic if and only if $\varepsilon_{1}=1$ and there exists an isomorphism $\delta:G_{1}\rightarrow G_{2}$ such that $\varepsilon_{2}\circ(\delta\times\delta)=1$.
\end{cor}

\begin{proof}
Indeed, by combining Propositions \ref{MAIN1} and \ref{main2}, we get directly the only if direction. The proof of the converse is similar to that of the previous corollary and then it is omitted.
\end{proof}

\bigskip

{\footnotesize \pn{\bf Noureddine Snanou }\; \\   {Department of Mathematics}, {Faculty of Sciences Dhar El Mahraz}, {Sidi Mohamed Ben Abdellah University}, {Fez, Morocco}\\ {\tt Email: noureddine.snanou@usmba.ac.ma}}\

\end{document}